\documentclass{amsart}

\usepackage{amssymb,mathtools}
\usepackage[british]{babel}
\usepackage[latin1]{inputenc}
\usepackage{enumitem}
\usepackage{url}

\usepackage{bussproofs}
\newenvironment{bprooftree}
  {\leavevmode\hbox\bgroup}
  {\DisplayProof\egroup}

\newtheorem{theorem}{Theorem}
\newtheorem{corollary}[theorem]{Corollary}
\newtheorem{lemma}[theorem]{Lemma}

\theoremstyle{definition}

\newcommand{\pra}{\mathbf{PRA}}
\newcommand{\pa}{\mathbf{PA}}
\newcommand{\con}{\operatorname{Con}}
\newcommand{\pr}{\operatorname{Pr_{\mathbf{PRA}}}}
\newcommand{\zstar}{\mathbf{Z^*}}

\title[A Note on Iterated Consistency and Infinite Proofs]{A Note on Iterated Consistency\\ and Infinite Proofs}
\thanks{This is a pre-print (before peer-review) of an article published in the Archive for Mathematical Logic. The final authenticated version is available online at \url{https://doi.org/10.1007/s00153-018-0639-y}. It can also be accessed via \url{https://rdcu.be/2YJI} (view only). The journal version contains some improvements over the present pre-print: In particular it emphasizes that the result is only valid for somewhat restrictive formalizations of infinite proofs. It also discusses the connection with Ulf Schmerl's paper ``Iterated reflection principles and the $\omega$-rule'', J.~Symb.~Logic \textbf{47}(4), 721-733 (1982)}
\author{Anton Freund}

\begin{document}

\begin{abstract}
Schmerl and Beklemishev's work on iterated reflection achieves two aims: It introduces the important notion of $\Pi^0_1$-ordinal, characterizing the $\Pi^0_1$-theorems of a theory in terms of transfinite iterations of consistency; and it provides an innovative calculus to compute the $\Pi^0_1$-ordinals for a range of theories. The present note demonstrates that these achievements are independent: We read off $\Pi^0_1$-ordinals from a Sch\"utte-style ordinal analysis via infinite proofs, in a direct and transparent way.
\end{abstract}

\maketitle

The central result of Schmerl's \cite{schmerl79} work on iterated reflection is the fine structure theorem: Over primitive recursive arithmetic ($\pra$), $\alpha$ iterations of (uniform) $\Pi_{n+1}$-reflection prove the same $\Pi_n$-theorems as $\omega^\alpha$ iterations of $\Pi_n$-reflection, for $n\geq 1$ and $\alpha>0$ (we will only be concerned with arithmetical formulas, so we usually write $\Pi_n$ rather than $\Pi^0_n$). Together with Kreisel and L\'evy's result \cite{kreisel68}, which states that Peano arithmetic ($\pa$) is equivalent to the collection of arithmetical reflection principles over $\pra$, this yields
\begin{equation*}
 \pa\equiv_{\Pi_1}\pra+\{\con_\alpha(\pra)\,|\,\alpha<\varepsilon_0\}.
\end{equation*}
Here $\equiv_{\Pi_1}$ indicates that the two theories prove the same $\Pi_1$-formulas. Furthermore, $\con_\alpha(\pra)$ is a $\Pi_1$-formula which expresses transfinite iterations of consistency over $\pra$. It will be convenient to consider iterations which are slightly stronger than Schmerl and Beklemishev's version at limit stages; a precise definition can be found below. Beklemishev \cite{beklemishev03} extends Schmerl's results (in particular weakening $\pra$ to Kalm\'ar elementary arithmetic) and coins the notion of $\Pi^0_1$-ordinal: This bounds the number of iterations of consistency that are provable in a given theory. In most natural cases these iterated consistency statements exhaust all $\Pi_1$-consequences of that theory, as in the example above. The $\Pi^0_1$-ordinal is particularly significant for theories with strong $\Pi_1$-axioms, which are not captured by other notions of proof-theoretic ordinal. For example, Beklemishev \cite[Corollary~7.11]{beklemishev03} shows that $\pa+\con(\pa)$ has the $\Pi^0_1$-ordinal~$\varepsilon_0\cdot 2$.

Besides being fascinating in itself, the fine structure theorem is useful to analyse a range of theories. It is particularly pertinent for theories such as $\pa+\con(\pa)$, which are naturally phrased in terms of reflection (or consistency). However, when it comes to stronger theories (e.g.\ of second order arithmetic) infinite proofs \`a la Sch\"utte are currently much more powerful. Also, infinite proofs can be useful to analyse restrictions on the cut rank or the proof length. It is thus important to ask whether one can infer $\Pi^0_1$-ordinals from a Sch\"utte-style ordinal analysis. Indeed, there are indirect ways to achieve this: An ordinal analysis via infinite proofs characterizes a theory in terms of transfinite induction. We can then invoke the connection between transfinite induction and iterated consistency established by Schmerl \cite[Section~2]{schmerl79}. Let us point out that this connection is not entirely trivial --- in particular Schmerl's proof uses the fine structure theorem, and it only works for ordinals of the form $\omega^{\omega^\alpha}$. Alternatively, we can use a Sch\"utte-style ordinal analysis to determine the provably total functions of a theory --- although this is again a non-trivial task, particularly for large ordinals, where fundamental sequences can be delicate. Beklemishev \cite[Section~3]{beklemishev03} connects the provably total functions to iterations of $\Pi_2$-reflection, and another application of the fine structure theorem yields the $\Pi^0_1$-ordinal. While these general connections exist, it would be desirable to directly read off $\Pi^0_1$-ordinals from infinite proofs. The present note demonstrates that this is very much possible: Assume that the $\Pi_1$-formula $\varphi$ has a cut-free $\omega$-proof of height $\alpha$, provably in $\pra$. We will see that this implies
\begin{equation*}
 \pra+\forall_{\gamma<\alpha}\con_\gamma(\pra)\vdash\varphi.
\end{equation*}
As a corollary we obtain the above characterization of the $\Pi_1$-consequences of Peano arithmetic. To minimize technical difficulty we will focus on infinite proofs that arise from an embedding of Peano arithmetic. Nevertheless, it will be clear that the same arguments apply to other proofs in $\omega$-logic, once they are suitably formalized in~$\pra$. Also, we focus on iterations of consistency (i.e.\ $\Pi_1$-reflection), which seems to be the most important case. It would be easy to adapt the same arguments to iterations of $\Pi_n$-reflection, in order to read off bounds on Beklemishev's $\Pi^0_n$-ordinals (see~\cite{beklemishev03}).

\section*{Iterated Consistency}

On an intuitive level, we define iterations of consistency over primitive recursive arithmetic by the recursion
\begin{equation*}
 \mathbf T_\alpha=\pra+\forall_{\gamma<\alpha}\con(\mathbf T_\gamma).
\end{equation*}
We point out that this deviates from the definition used by Beklemishev, who takes
\begin{equation*}
 \mathbf T'_\alpha=\pra+\{\con(\mathbf T'_\gamma)\,|\,\gamma<\alpha\}.
\end{equation*}
The difference is just an index shift at limit stages: As $\mathbf T'_{\alpha+1}$ entails $\con(\mathbf T'_\alpha)$ it proves that all $\Pi_1$-consequences of $\mathbf T'_\alpha$ are true, which gives $\mathbf T'_{\alpha+1}\vdash\forall_{\gamma<\alpha}\con(\mathbf T'_\gamma)$. The stronger limit step will make for more appealing bounds in our context. More formally, the definition of the theories $\mathbf T_\alpha$ depends on an arithmetization of the relevant ordinals. In this note we work with the usual notation system for $\varepsilon_0$. Following \cite[Section~2]{beklemishev03}, it will be most convenient to formalize the relation ``$\varphi$ is provable in $\mathbf T_\alpha$'': Using an arithmetization $\pr(\cdot)$ of provability in $\pra$ the fixed point lemma yields a formula $\Box_\alpha(\varphi)$ with
\begin{equation}\label{eq:def-prov-in-pra-alpha}
 \pra\vdash\Box_\alpha(\varphi)\leftrightarrow\pr(\forall_{\gamma<\dot\alpha}\neg\Box_{\gamma}(0=1)\rightarrow\varphi).
\end{equation}
This very equivalence shows that $\Box_\alpha(\varphi)$ is $\Sigma_1$ in $\pra$. Thus
\begin{equation*}
 \con_\alpha(\pra):\equiv\neg\Box_\alpha(0=1)
\end{equation*}
is a $\Pi_1$-formula. Negating both sides of (\ref{eq:def-prov-in-pra-alpha}) we see
\begin{equation*}
 \pra\vdash\con_\alpha(\pra)\leftrightarrow\con(\pra+\forall_{\gamma<\dot\alpha}\con_\gamma(\pra)).
\end{equation*}
This suggests that
\begin{equation*}
 \mathbf T_\alpha:=\pra+\forall_{\gamma<\alpha}\con_\gamma(\pra)
\end{equation*}
captures the above intuition. Further justification comes from the result that equivalences such as (\ref{eq:def-prov-in-pra-alpha}) determine the relation $\Box_\alpha(\varphi)$ completely (see \cite[Lemma~2.3]{beklemishev03}). It may be astonishing that this holds over $\pra$, where we have limited access to transfinite induction. Instead, the result relies on a principle of ``reflexive induction'', due to Schmerl and Girard (see \cite{schmerl79}). Beklemishev states a slightly strengthened version, which will be important later:

\begin{lemma}\label{lem:reflexive-induction}
 For any formula $\varphi(\alpha)$ in the language of $\pra$, if we have
\begin{equation*}
 \pra\vdash\forall_\alpha(\pr(\forall_{\gamma<\dot\alpha}\varphi(\gamma))\rightarrow\varphi(\alpha))
\end{equation*}
--- we say that $\varphi$ is reflexively progressive --- then we can conclude
\begin{equation*}
 \pra\vdash\forall_\alpha\varphi(\alpha).
\end{equation*}
\end{lemma}
\begin{proof}
We cite the short argument from \cite[Lemma~2.4]{beklemishev03}: To conclude by L\"ob's theorem it suffices to establish
\begin{equation*}
 \pra\vdash\pr(\forall_\alpha\varphi(\alpha))\rightarrow\forall_\alpha\varphi(\alpha).
\end{equation*}
Working in $\pra$, assume $\pr(\forall_\alpha\varphi(\alpha))$. In particular $\pr(\forall_{\gamma<\dot\alpha}\varphi(\gamma))$ holds for any $\alpha$. By assumption this implies $\varphi(\alpha)$. Since $\alpha$ was arbitrary we have~$\forall_\alpha\varphi(\alpha)$, as required.
\end{proof}

We will also need the following easy observation:

\begin{lemma}\label{lem:box-alpha-pi1-reflection}
 If the formula $\varphi(x_1,\dots,x_n)$ is $\Pi_1$ in $\pra$ then we have
\begin{equation*}
 \pra\vdash\forall_\alpha(\con_\alpha(\pra)\land\Box_\alpha(\varphi(\dot x_1,\dots,\dot x_n))\rightarrow\varphi(x_1,\dots,x_n)).
\end{equation*}
\end{lemma}
\begin{proof}
 Arguing in $\pra$, consider an arbitrary ordinal $\alpha$. To establish the contrapositive of the claim, assume that we have $\neg\varphi(x_1,\dots,x_n)$. By $\Sigma_1$-completeness we get $\pr(\varphi(\dot x_1,\dots,\dot x_n)\rightarrow 0=1)$. Also assume $\Box_\alpha(\varphi(\dot x_1,\dots,\dot x_n))$, and observe that this means $\pr(\forall_{\gamma<\dot\alpha}\neg\Box_\gamma(0=1)\rightarrow\varphi(\dot x_1,\dots,\dot x_n))$. Combining the two implications we obtain $\pr(\forall_{\gamma<\dot\alpha}\neg\Box_\gamma(0=1)\rightarrow0=1)$. This amounts to $\Box_\alpha(0=1)$, which is equivalent to $\neg\con_\alpha(\pra)$, completing the proof of the contrapositive.
\end{proof}

\section*{Bounds from Infinite Proofs}

Let us briefly review the ordinal analysis of Peano arithmetic via infinite proofs, and its formalization in $\pra$. We will follow the approach of Buchholz \cite{buchholz91}, where the reader can find all missing details. The infinitary proof system is based on (Tait-style) sequent calculus: Rather than single formulas one derives finite sets (``sequents''), which are to be read disjunctively. Thus, the sequent $\Gamma\equiv\{\varphi_1,\dots,\varphi_n\}$ is interpreted as the disjunction $\bigvee\Gamma\equiv\varphi_1\lor\dots\lor\varphi_n$. As usual we write $\Gamma,\varphi$ rather than $\Gamma\cup\{\varphi\}$. The characteristic feature of our infinite calculus is the $\omega$-rule
\begin{equation*}
\begin{bprooftree}
\AxiomC{$\Gamma,\varphi(0)$}
\AxiomC{$\Gamma,\varphi(1)$}
\AxiomC{$\cdots$}
\RightLabel{,}
\TrinaryInfC{$\Gamma,\forall_n\varphi(n)$}
\end{bprooftree}
\end{equation*}
with a premise for each numeral. In other words, if we know that the disjunction $\bigvee\Gamma\lor\varphi(n)$ holds for each number~$n$ then we can infer $\bigvee\Gamma\lor\forall_n\varphi(n)$. Similarly, there are rules to introduce propositional connectives and existential quantifiers, in these cases with finitely many premises. As axioms we take true prime formulas. To derive a formula we combine these rules into a proof tree. In the presence of the infinitary $\omega$-rule these trees will generally have infinite height. The reader may observe that any true arithmetical formula can be derived with finite height --- but this is not much use if we work in $\pra$ or another weak theory, which recognizes few formulas as true. What $\pra$ does know (see below) is that proofs in Peano arithmetic can be translated into infinite proofs. The resulting proof trees have height below $\omega\cdot 2$. At the same time, this embedding introduces cut rules
\begin{equation*}
\begin{bprooftree}
\AxiomC{$\Gamma,\varphi$}
\AxiomC{$\Gamma,\neg\varphi$}
\RightLabel{.}
\BinaryInfC{$\Gamma$}
\end{bprooftree}
\end{equation*}
Note that the cut formula $\varphi$ may be very complex, even if we know that the end-sequent $\Gamma$ is very simple. This can be problematic when we check properties by induction over proofs. The famous method of cut elimination removes that obstacle: It allows us to reduce the complexity of cut formulas, increasing the height of the proof by a power of $\omega$ in each step. In the end we obtain a cut-free proof with height below~$\varepsilon_0$. From this proof we can read off the desired bounds. There are several ways to formalize infinite proofs in $\pra$. One option would be to work with primitive recursive proof trees, represented by numerical codes. We choose the particularly elegant approach of Buchholz \cite{buchholz91}: The idea is to name proofs by the role they play in the cut elimination process. For each (finite) proof $d$ in Peano arithmetic one introduces a constant symbol $[d]$, which names the embedding of~$d$ into the infinite system. Furthermore, one adds a unary function symbol $E$: If~$h$ is a name for an infinite proof tree then the term $E h$ names the result of cut elimination (with cut rank reduced by one). Auxiliary function symbols $I_{k,\varphi}$ and $R_\psi$ refer to inversion and reduction, the usual ingredients of cut elimination. The resulting set of terms is denoted by $\zstar$. Crucially, there is a primitive recursive function $\mathfrak s:\mathbb N\times\zstar\rightarrow\zstar$ which computes codes for the immediate subtrees of a proof tree. For example, if the last rule of $h$ (also read off by a primitive recursive function) is not a cut then we have $\mathfrak s_n(E h)=E\mathfrak s_n(h)$. Intuitively this means that we apply cut elimination to the $n$-th subtree of $h$ in order to get the $n$-th subtree of~$Eh$. Officially, the equation $\mathfrak s_n(E h)=E\mathfrak s_n(h)$ is part of the definition of $\mathfrak s$, which goes by recursion over the terms in $\zstar$. Similarly, we have functions to read off the end sequent, the cut rank, and the ordinal height of (the proof represented by) a given term. We should mention that it is crucial to extend the infinitary system by a ``repetition rule'', which simply repeats its premise: It permits us to ``call'' a proof even if we cannot immediately determine its last rule. The point is that Buchholz' approach allows a very smooth formalization in $\pra$: We simply work with a system of terms, without official reference to their interpretation as infinite trees. In $\pra$ we can show that the proof terms are ``locally correct''. Amongst other things, this means that the ordinal height of $\mathfrak s_n(h)$ is smaller than the ordinal height of~$h$. Having explained this approach we can revert to more traditional notation: By $h\vdash^\alpha_0\Gamma$ we express that $h\in\zstar$ codes a cut-free proof of (a sub-sequent of)~$\Gamma$ with ordinal height~$\alpha$. We write $\vdash^\alpha_0\Gamma$ to indicate that this holds for some term~$h\in\zstar$.

When we try to track usual (i.e.~finite) provability through an infinite proof, the \mbox{$\omega$-rule} poses an obstruction: From $\pra\vdash\varphi(n)$ for all $n\in\mathbb N$ we cannot infer $\pra\vdash\forall_n\varphi(n)$. On the other hand, $\pra\vdash\forall_n\pr(\varphi(\dot n))$ does imply $\pra+\con(\pra)\vdash\forall_n\varphi(n)$, provided that $\varphi$ is a $\Pi_1$-formula. This explains why iterated consistency is needed. Iterations of consistency are important for a second reason as well: They will make it possible to use the reflexive induction principle. These ideas lead to the main result of the present note:

\begin{theorem}
 Provably in $\pra$, we have
\begin{equation*}
 \vdash^\alpha_0\Gamma\quad\Rightarrow\quad\Box_\alpha(\bigvee\Gamma)
\end{equation*}
for any sequent $\Gamma$ that consists of $\Pi_1$-formulas only.
\end{theorem}
\begin{proof}
 We argue by reflexive induction on $\alpha$, as established in Lemma~\ref{lem:reflexive-induction}. Working in~$\pra$, we may thus assume
\begin{equation*}
 \pr(\forall_{\gamma<\dot\alpha}\forall_{h\in\zstar}(h\vdash^\gamma_0\Gamma\land\text{``$\Gamma$ consists of $\Pi_1$-formulas''}\rightarrow\Box_\gamma(\bigvee\Gamma)))
\end{equation*}
in order to deduce the claim for $\alpha$. The antecendent of that claim provides $h\vdash^\alpha_0\Gamma$ for some $h\in\zstar$, where $\Gamma$ is a sequent of $\Pi_1$-formulas. We must deduce $\Box_\alpha(\bigvee\Gamma)$. Let us assume that $h$ ends in an $\omega$-rule that introduces the formula $\forall_n\varphi(n)\in\Gamma$ ---~in all other cases the argument is similar and easier. By $\Sigma_1$-completeness we can recover the assumptions inside $\pra$, i.e.\ we get
\begin{equation*}
 \pr(\dot h\vdash^{\dot\alpha}_0\dot\Gamma\land\text{``$\dot h$ ends in an $\omega$-rule that introduces $\forall_n\varphi(n)$''}).
\end{equation*}
Recall that $\mathfrak s_n(h)\in\zstar$ denotes the $n$-th immediate subtree of the infinite proof denoted by $h$. By the local correctness of $\zstar$, this subtree deduces the premise~$\Gamma,\varphi(n)$ of the $\omega$-rule, with ordinal height $\alpha_n<\alpha$. Applying this argument in the meta theory $\pra$ would only give the subtrees $\mathfrak s_n(h)$ for ``standard'' numbers $n$. This is not enough for our purpose, but the same argument in the object theory does yield
\begin{equation*}
 \pr(\forall_n\exists_{\alpha_n<\dot\alpha}\,\mathfrak s_n(\dot h)\vdash^{\alpha_n}_0\dot\Gamma,\dot\varphi(n)).
\end{equation*}
Note that $\varphi(n)$ is a $\Pi_1$-formula (indeed a $\Delta_0$-formula), since $\forall_n\varphi(n)$ occurs in $\Gamma$. Thus the reflexive induction hypothesis implies
\begin{equation*}
 \pr(\forall_n\exists_{\alpha_n<\dot\alpha}\,\Box_{\alpha_n}(\bigvee\Gamma\lor\varphi(\dot n))).
\end{equation*}
Formalizing Lemma~\ref{lem:box-alpha-pi1-reflection} in $\pra$ we get
\begin{equation*}
 \pr(\forall_{\gamma<\dot\alpha}\con_\gamma(\pra)\rightarrow\forall_n(\bigvee\Gamma\lor\varphi(n))).
\end{equation*}
In view of $\forall_n\varphi(n)\in\Gamma$ one obtains
\begin{equation*}
 \pr(\forall_{\gamma<\dot\alpha}\con_\gamma(\pra)\rightarrow\bigvee\Gamma).
\end{equation*}
By equivalence (\ref{eq:def-prov-in-pra-alpha}) this amounts to $\Box_\alpha(\bigvee\Gamma)$, as required.
\end{proof}

The result cited in the introduction is an easy consequence:

\begin{corollary}[Schmerl]
 We have
\begin{equation*}
 \pa\equiv_{\Pi_1}\pra+\{\con_\alpha(\pra)\,|\,\alpha<\varepsilon_0\}.
\end{equation*}
\end{corollary}
\begin{proof}
First, assume that $\varphi$ is a $\Pi_1$-theorem of Peano arithmetic. By embedding and cut elimination (see \cite{buchholz91}) we get $\vdash^\alpha_0 \varphi$ for some ordinal $\alpha<\varepsilon_0$, provably in~$\pra$. In this situation, the previous theorem yields $\pra\vdash\Box_\alpha(\varphi)$. Using Lemma~\ref{lem:box-alpha-pi1-reflection} we can infer \mbox{$\pra+\con_\alpha(\pra)\vdash\varphi$}, as required. In fact, equivalence~(\ref{eq:def-prov-in-pra-alpha}) and the soundness of $\pra$ give the sharper bound
\begin{equation*}
 \pra+\forall_{\gamma<\alpha}\con_\gamma(\pra)\vdash\varphi,
\end{equation*}
as promised in the introduction. Conversely, Peano arithmetic proves $\con_\alpha(\pa)$ by transfinite induction up to any fixed $\alpha<\varepsilon_0$: Assume $\forall_{\beta<\gamma}\con_\beta(\pra)$ by induction hypothesis. The contrapositive of $\Sigma_1$-reflection over $\pra$, which is provable in $\pa$, yields $\con(\pra+\forall_{\beta<\dot\gamma}\con_\beta(\pra))$, or equivalently $\con_\gamma(\pra)$.
\end{proof}

\bibliographystyle{alpha}
\bibliography{Iterated-Consistency_Freund.bib}

\begin{thebibliography}{Buc91}

\bibitem[Bek03]{beklemishev03}
Lev Beklemishev.
\newblock Proof-theoretic analysis by iterated reflection.
\newblock {\em Archive for Mathematical Logic}, 42(6):515--552, 2003.

\bibitem[Buc91]{buchholz91}
Wilfried Buchholz.
\newblock Notation systems for infinitary derivations.
\newblock {\em Archive for Mathematical Logic}, 30:277--296, 1991.

\bibitem[KL68]{kreisel68}
Georg Kreisel and Azriel Lévy.
\newblock Reflection principles and their use for establishing the complexity
  of axiomatic systems.
\newblock {\em Zeitschrift für mathematische Logik und Grundlagen der
  Mathematik}, 14:97--142, 1968.

\bibitem[Sch79]{schmerl79}
Ulf~R.\ Schmerl.
\newblock A fine structure generated by reflection formulas over primitive
  recursive arithmetic.
\newblock In M.\ Boffa, D.\ van Dalen, and K.\ MacAloon, editors, {\em Logic
  Colloquium `78}, pages 335--350, 1979.

\end{thebibliography}

\end{document}